\documentclass{article}
\usepackage {amssymb}
\usepackage {amsmath}
\usepackage {amsthm}
\usepackage{graphicx}
\usepackage{subfigure}
\usepackage {amscd}
\usepackage{xypic}
\usepackage[colorlinks, citecolor=red]{hyperref}
\newtheorem{lem}{Lemma}
\newtheorem{thm}{Theorem}
\newtheorem{defn}{Definition}

\newcommand{\sX}{{\mathcal{X}}}
\newcommand{\sY}{{\mathcal{Y}}}

\newcommand{\sm}[0]{\operatorname{sm}}
\newcommand{\orb}[0]{\operatorname{orb}}
\newcommand{\orp}[0]{\operatorname{par}}
\newcommand{\sddbar}{{\sqrt{-1}\partial\bar{\partial}}}

\newcommand{\reg}{{\rm reg}}
\newcommand{\sing}{{\rm sing}}
\newcommand{\KE}{{\rm KE}}

\begin{document}

\title{Orbifold regularity of weak K\"{a}hler-Einstein metrics}
\author{Chi Li, Gang Tian}

\author{Chi Li\thanks{Supported partially by an NSF grant. Email: chi.li@stonybrook.edu}\\
Stony Brook University\\[5pt]
Gang Tian\thanks{Supported partially by
NSF and NSFC grants. Email: tian@math.princeton.edu}\\
Beijing University and Princeton University
}
\date{}
\maketitle

\section{Introduction}
In the resolution of the YTD conjecture on the existence of K\"{a}hler-Einstein metrics on Fano manifolds (see \cite{Tian12} and also \cite{CDS}), a crucial tool is a compactness result.
%There are two important aspects of this progress. The first concerns the study of conical K\"{a}hler-Einstein metrics and logarithmic K-stability.
%The other is about a
%compactness results for a sequence of (conical) K\"{a}hler-Einstein manifolds.
%In this note, we make a remarks about the regularity of the limit metric space.
%\section{Regularity on the orbifold locus}
%\section{Introduction}
%The resolution of Yau-Tian-Donaldson conjecture depends essentially on a compactness result.
In its simplest form, this result says that the Gromov-Hausdorff limit of a sequence of smooth K\"{a}hler-Einstein manifolds $(X_i, \omega_{i,\KE})$ is a normal Fano variety $X:=X_\infty$ with klt singularities and that there is a weak K\"{a}hler-Einstein metric $\omega_{\infty, \KE}^w$ on $X_{\infty}$. The existence of a Gromov-Hausdorff limit follows from Gromov's compactness theorem. So the important information in this statement is about the regularity of $X_\infty$. It was the second author (\cite{Tian90}, \cite{Tian09}, see also \cite{Li12a}) who first pointed out the route to prove that $X_\infty$ is an algebraic variety is to establish a so-called partial $C^0$-estimate. He demonstrated in \cite{Tian90} how to achieve this when the complex dimension $n$ is equal to $2$ by showing that a sequence of K\"{a}hler-Einstein surfaces converges to a Fano orbifold with a smooth orbifold K\"{a}hler-Einstein metric. Note that when $n=2$, klt singularities are nothing but quotient singularities or orbifold singularities. Two key ingredients to prove the partial $C^0$-estimate in dimension 2 are orbifold compactness result of Einstein 4-manifolds and H\"{o}rmander's $L^2$-estimates. 

Recently, Donaldson-Sun \cite{DS12} and the second author \cite{Tian13} generalized the partial $C^0$-estimate to higher dimensional K\"ahler-Einstein manifolds. Here they need to rely on compactness results of higher dimensional Einstein manifolds developed by Cheeger-Colding and Cheeger-Colding-Tian (see \cite{CCT} and the reference therein).  Compared to the complex dimension 2 case, the second author also conjectured that $\omega_{\infty, \KE}$ is a smooth orbifold metric away from analytic subvarieties of complex codimension 3. Note that in \cite{CCT}, it was proved that the (metric) singular set of $X_{\infty}$ has complex codimension at least 2.

It can be shown that, by partial $C^0$-estimate, there is a uniform $C^2$-estimate of the potential of $\omega^{w}_{\infty,\KE}$ on $X_{\infty}^{\reg}$. Then the Evans-Krylov theory or Calabi's 3rd derivative estimate allows one to show that $\omega_{\infty, \KE}^w$ is smooth on $X_{\infty}^{\reg}$ (see \cite{Tian90}, \cite{DS12}, \cite{Tian12}).
Alternatively using P\v{a}un's Laplacian estimate in \cite{Paun} and Evans-Krylov theory, Berman-Boucksom-Eyssidieux-Guedj-Zeriahi \cite{BBEGZ} %and Donaldson-Sun \cite{DS12} proved that
showed directly that any weak K\"{a}hler-Einstein metric $\omega_{\KE}^w$ on a klt Fano variety $X_{\infty}$ is smooth on $X_{\infty}^{\reg}$.
The purpose of this note is to answer the question by the second author about the regularity of $\omega^{w}_{\KE}$ on the orbifold locus $X_{\infty}^{\orb}$ of $X_{\infty}$.
First, if $(X, -K_X)$ is a klt Fano variety, then by \cite[Proposition 9.3]{GKKP} there exists a closed subset $Z\subset X$ with
$codim_XZ\ge 3$ such that $X\backslash Z$ has quotient singularities.
So we just need to show the following regularity result. For the definition of weak K\"{a}hler-Einstein metric, see Definition \ref{wkKE}.
\begin{thm}\label{orbreg}
Assume that $\omega_{\KE}^w$ is a weak K\"{a}hler-Einstein metric on $X_{\infty}$. Then $\omega_{\KE}^w$ is a smooth orbifold metric on $X_{\infty}^{\orb}$.
\end{thm}
Our proof now uses the existence of an orbifold resolution, i.e., Theorem \ref{orbresolve} which is proved by algebraic method. However, we believe that
it is not necessary. There should be a purely differential geometric proof of
Theorem \ref{orbreg} which does not rely on Theorem \ref{orbresolve}. In a subsequent paper, we will analyze further structures of singularities of
higher codimension. We believe that our analysis can be used to yield a complete understanding of the singularity for any 3-dimensional weak K\"ahler-Einstein metrics.

\section{Regularity on the orbifold locus}

From now on we will denote by $X$ any $\mathbb{Q}$-Fano variety with klt singularities. Assume $\iota: X\rightarrow \mathbb{P}^N$ is an embedding given by the linear system $|-m K_X|$ for $m>0\in\mathbb{Z}$
sufficiently large and divisible.
Let $h_0=(\iota^*{h_{FS}})^{1/m}$ be the pull back of the Fubini-Study Hermitian metric $h_{FS}$ on $\mathcal{O}_{\mathbb{P}^N}(1)$ normalized to be a Hermitian metric
on $-K_X$. The Chern curvature form of $h_0$ is
\[
\omega_0=-\sddbar\log h_0
\]
which is a positive $(1,1)$-current on $X$. $\omega_0$ is a smooth positive definite $(1,1)$-form on $X^{\reg}$. However, on the singular locus $X^{\sing}$, $\omega_0$ in general is not canonically related to the local structure of $X$. Assume $p\in X^{\orb}$ is a quotient singularity. By this, we mean that there exists a small neighborhood $\mathcal{U}_p$ which is isomorphic to a quotient of a smooth manifold by a finite group. In other words, there exists a branched covering map
$\tilde{\mathcal{U}}_p\rightarrow \tilde{\mathcal{U}}_p/G\cong\mathcal{U}_p$.  The lifting of metric $\omega_0$ to the cover $\tilde{\mathcal{U}}_p$ in general is degenerate.

Now we define an adapted volume for on $X$ by
\[
\Omega=|v^*|_{h_0}^{2/m}(v\wedge \bar{v})^{1/m}.
\]
Here $v$ is any local generator of $\mathcal{O}(mK_X)$ and $v^*$ is the dual generator of $\mathcal{O}(-mK_X)$. The K\"{a}hler-Einstein equation
\begin{equation}
Ric(\omega_\phi)=\omega_\phi.
\end{equation}
can be transformed into a complex Monge-Amp\`{e}re equation:
\begin{equation}\label{CMA}
(\omega_0+\sddbar\phi)^n=%\frac{e^{-\phi}\Omega}{V^{-1}\int_Xe^{-\phi}\Omega}=
e^{-\phi}\Omega.
\end{equation}
\begin{defn}\label{wkKE}
A weak solution to the \eqref{CMA} is a bounded function $\phi\in L^{\infty}(X)\cap PSH(X,\omega)$ satisfying \eqref{CMA} in the sense of pluripotential theory.
\end{defn}
Let's first recall the method to prove the regularity of $\phi$ on $X^{\reg}$ following \cite{BBEGZ}. One first chooses a resolution $\pi: \tilde{X}\rightarrow X$ with simple normal crossing exceptional divisor $E=\pi^{-1}(X^{\sing})$ such that $\pi$ is an isomorphism over $X^{\reg}$. Then we can
pull back the equation \eqref{CMA} to $\tilde{X}$ and get:
\begin{equation}
(\pi^*\omega_0+\sddbar\psi)^n=e^{-\psi}\pi^*\Omega.
\end{equation}
On the other hand we can write:
\[
K_{\tilde{X}}=\pi^*K_{X}+\sum_{i=1}^ra_i E_i -\sum_{j=1}^s b_j F_j,
\]
such that $E=\cup_{i=1}^r E_i \bigcup \cup_{j=1}^s F_j$ and $a_i>0$, $b_j>0$.
The klt property implies: $a_i>0$, and $0<b_j<1$. Analytically, choosing a smooth K\"{a}hler metric $\eta$ on $\tilde{X}$, there exists $f\in C^{\infty}(\tilde{X})$ such that:
\[
\pi^*\Omega=e^{f}\frac{\prod_{i=1}^r |s_i|^{2a_i}}{\prod_{j=1}^s|\sigma_j|_{}^{2b_j}}\eta^n.
\]
where $s_i$ and $\sigma_j$ are defining sections of $E_i$ for $F_j$ respectively and $|s_i|^2$ and $|\sigma_j|^2$ are some fixed hermitian norms of them. So we have:
\begin{equation}\label{degMA}
(\pi^*\omega_0+\sddbar\psi)^n=e^{-\psi+f+\sum_{i}a_i \log |s_i|^2-\sum_j b_j\log |\sigma_j|^2}\eta^n=e^{\psi_{+}-\psi_{-}}\eta^n,
\end{equation}
Here we have denoted
\[
\psi_{+}=f+\sum_{i}a_i\log|s_i|^2, \quad \psi_{-}=\psi+\sum_{j}b_j\log|\sigma_j|^2.
\]
It's easy to see that they satisfy the quasi-plurisubharmonic condition:
\begin{equation}\label{quapsh}
\sqrt{-1}\partial\bar{\partial}\psi_{+}\ge -C \eta, \sqrt{-1}\partial\bar{\partial}\psi_{-}\ge -C\eta,
\end{equation}
for some uniform constant $C>0$.
To get Laplacian estimate of $\psi$ away from $Z=\pi^{-1}(X_\infty)$, we can first regularize \eqref{degMA} to
\begin{equation}\label{regMA}
(\omega_\epsilon+\sddbar\psi_{\epsilon})^n=e^{\psi_{+,\epsilon}-\psi_{-,\epsilon}}\eta^n.
\end{equation}
where $\omega_\epsilon=\pi^*\omega_0-\epsilon\theta_{E}$ is a K\"{a}hler metric on $\tilde{X}$, and $\psi_{\pm,\epsilon}\in C^{\infty}(\tilde{X})$ converges to $\psi_{\pm}$ in $L^p(\tilde{X})\cap L^{\infty}(\tilde{X}\backslash Z)$ for some $p>1$.
Using \eqref{quapsh} and cleverly modifying the $C^2$-estimate of Aubin-Yau-Siu,
P\v{a}un \cite{Paun} proved the Laplacian estimate for the solution $\psi_\epsilon$ away from $Z$. More precisely, for any compact set $K\subset \tilde{X}\backslash Z$, there exists a constant $A=A(\|\psi\|_{\infty},K)$, such that
\[
\Delta_{\eta}\psi_{\epsilon} \le A(\|\psi\|_{\infty}, K) e^{-\psi_{-}}.
\]
From this estimate, we know that the right-hand side of \eqref{regMA} is uniformly $C^{1,\alpha}$ % \eqref{deMA} is locally $C^{1,\alpha}$
on $\tilde{X}\backslash Z$. By Evan-Krylov's theory (\cite{Blocki}), we know that $\psi_{\epsilon}$ is uniformly $C^{2,\alpha}$ and hence by bootstrapping, $C^{k,\alpha}$ on $\tilde{X}\backslash Z$. Now because $\psi_{\epsilon}$ converges to $\psi$ in $C^{k}$ norm uniformly away from $Z$,
we get that $\psi$ is smooth on $\tilde{X}\backslash Z$.

One can also prove the regularity on $X^{\reg}$ with the help of K\"{a}hler-Ricci flow. Starting from the work in \cite{CTZ}, this idea has been used several times in the literature to prove the regularity of weak solutions to complex Monge-Amp\`{e}re equations. Recall that the K\"{a}hler-Ricci flow is a solution to the following equation:
\begin{equation}
\left\{
\begin{array}{l}
\frac{\partial\omega_t}{\partial t}=-Ric(\omega_t)+\omega_t;\\
\omega(0)=\omega_{\phi_0}.
\end{array}
\right.
\end{equation}
As in the elliptic case, this equation can be transformed into the following Monge-Amp\`{e}re flow
\begin{equation}\label{MAflow1}
\left\{
\begin{array}{l}
\frac{\partial \phi}{\partial t}=\log\frac{(\omega_0+\sddbar\phi)^n}{\Omega}+\phi;\\
\phi(0,\cdot)=\phi_0.
\end{array}
\right.
\end{equation}

To define a solution to this Monge-Amp\`{e}re flow on the singular variety $X$, % first choose a smooth resolution of singularities $\pi: \tilde{X}\rightarrow X$ with exceptional divisor $E$.
Song-Tian \cite{ST09} pulled up the flow equation in \eqref{MAflow1} to $\tilde{X}$ to get:
\begin{equation}\label{MAflowb}
\left\{
\begin{array}{l}
\frac{\partial \tilde{\phi}}{\partial t}=\log\frac{(\pi^*\omega_0+\sddbar\tilde{\phi})^n}{\pi^*\Omega}+\tilde{\phi};\\
\\
\tilde{\phi}(0,\cdot)=\pi^*\phi_0.
\end{array}
\right.
\end{equation}
\begin{thm}[\cite{ST09}]\label{ST09}
Let $\phi_0\in PSH_p(X,\omega_0)$ for some $p>1$. Then the Monge-Amp\`{e}re flow \eqref{MAflowb} on $\tilde{X}\backslash E$ has a unique solution $\tilde{\phi}\in C^{\infty}((0,T_0)\times\tilde{X}\backslash E)\cap C^0([0,T_0)\times\tilde{X}\backslash E)$ such that for all $t\in [0, T_0)$, $\tilde{\phi}(t,\cdot)\in L^{\infty}(\tilde{X})\cap PSH(\tilde{X},\pi^*\omega_0)$.
\end{thm}
Since $\tilde{\phi}$ is constant along (connected) fibre of $\pi$, $\tilde{\phi}$ descends to a solution $\phi \in C^{\infty}((0,T_0)\times X^{\reg})\cap C^0([0, T_0)\times X^{\reg})$ of the Monge-Amp\`{e}re
flow \label{MAflow2}.

Now suppose $\omega^{w}_{\KE}=\omega_0+\sddbar\phi^{w}_{\KE}$ is a weak solution to the equation \eqref{CMA}.  If one can prove that the solution $\phi(t)$ to \eqref{MAflow1} with the initial condition $\phi(0)=\phi^w_{\KE}$ is stationary, then it follows from Theorem \eqref{ST09} that $\omega^{w}_{\KE}$ is smooth on $X_\infty^{\reg}$.  The idea to prove stationarity in \cite{CTZ} is to show that the energy functional is decreasing along the flow solution $\phi(t)$ and to use the uniqueness of weak K\"{a}hler-Einstein metrics. These are indeed true in the current case by the work of \cite{BBEGZ}.

To prove Theorem \ref{orbreg}, the main observation is that the above arguments can be used to prove the regularity of $\omega_{\KE}^{w}$ on $X^{\orb}$ as long as one can find a partial resolution by orbifolds: $\pi^{\orp}: X^{\orp}\rightarrow X$. Indeed, by the next section, there exist orbifold (partial) resolutions. If $\pi^{\rm par}: X^{\rm par}\rightarrow X$ is an orbifold resolution, then we can write:
\[
K_{X^{\rm par}}=(\pi^{\rm par})^* K_{X}+\sum_i^r {a_i}E_i-\sum_{j=1}^s b_j F_j,
\]
where $E=\cup_{i=1}^r E_i \bigcup \cup_{j=1}^s F_j $ is now a simple normal crossing divisor within orbifold category (in the sense of Satake \cite{Sat,Sat2}).  The klt property of $X$ again implies $a_i>0$ and $0<b_i<1$. Then the similar argument as in the proof of regularity of $\omega_{\KE}^w$ on $X_\infty^{\reg}$ carries over to the orbifold setting to prove the orbifold regularity of $\omega_{\KE}^w$ on $X^{\rm orb}$.
 %Indeed, we can use the above arguments as for the regular part because we have:
%\begin{prop}[{\cite[Proposition 5.20]{KM}}]
%Assume $\tilde{X}\rightarrow X$ is a index 1 canonical covering. $X$ is klt if and only if $\tilde{X}$ is klt.
%\end{prop}

Note that it was already observed in \cite[Section 4.3]{ST09} that if $X$ has only orbifold singularities, then the K\"{a}hler-Ricci flow smooths out initial metric to become genuine smooth {\it orbifold} metric immediately when $t>0$.

\section{Orbifold partial resolution}

The results in this section were communicated to us by Chenyang Xu.

\begin{lem}[Resolution of Deligne-Mumford stacks]\label{l-rdm} Let $\sX$ be an integral Deligne-Mumford stack which is of finite type over $\mathbb{C}$. Then there exists a birational proper representable morphism $g^{\sm}:\sX^{\sm}\to \sX$ from a smooth Deligne-Mumford stack $\sX^{\sm}$. Furthermore, we can assume that $g^{\sm}$ is isomorphic over the smooth locus of $\sX$, and the exceptional locus of $g^{\sm}$ is a normal crossing divisorial closed substacks of $\sX^{\sm}$.
\end{lem}
\begin{proof}
This follows from the functoriality property of resolution of singularities (see \cite{Wl}, \cite{Kollar}, \cite{BM},\cite{Temkin}).
%To be added.
\end{proof}
\begin{lem}[Blow up the indeterminacy locus]\label{l-bldm} Let $X$ be a projective scheme. Let $\sX$ be a normal Deligne-Mumford stack with a dense open set $f_U:U\subset \sX$, such that $U$ admits a morphism $U\to X$. Then we can blow up an ideal $\mathcal{I}\subset \mathcal{O}_{\sX}$ to obtain a Deligne-Mumford stack $\tilde{\sX}$ such that $\tilde{\sX}\to \sX$ is isomorphic over $U$ and $f_U$ extends to a morphism $f: \tilde{\sX}\to X$.
\end{lem}
\begin{proof} We can replace $X$ by $\mathbb{P}^N$. %Let $L_U$ be the pull back of $\mathcal{O}(1)$ to $U$.
Let $D\subset U$ be the pull back of a hyperplane section $H$ which does not vanish along $U$, and we let $\mathcal{I}\subset \mathcal{O}_{\sX}$ be the ideal of the closure of $D$ in $\sX$. Then the rest of the proof follows from the cases for schemes as in \cite[II.7.17.3]{Har}.
\end{proof}

\begin{thm}\label{orbresolve}
Let $X$ be a quasi-projective normal variety. Let $X^{\orb}$ be the locus where $X$ only has orbifold singularity. Then there exists $f^{\orp}:X^{\orp}\to X$ a proper birational morphism, such that $X^{\orp}$ only has quotient singularity and $f^{\orp}$ is an isomorphic  over $X^{\orb}$.
\end{thm}
\begin{proof}After taking the closure of $X\subset \mathbb{P}^N$, we can assume $X$ is projective.

By \cite[2.8]{Vistoli89}, we know there is a smooth Deligne-Mumford stack $\sX^0$ whose coarse moduli space is $X^{\orb}$.
It follows from \cite[Theorem 4.4]{Kresch09} that $\sX^0=[Z/G]$  for some quasi-projective scheme $Z$ and linear algebraic group $G$. Actually, $Z$ can be taken as the frame bundle of $X^{\orb}$ and $G=GL_n(\mathbb{C})$. Then by \cite[Theorem 5.3]{Kresch09}, there is a proper Deligne-Mumford stack $\sX$, such that $\sX^0\subset \sX$ is a dense open set.
%Applying Lemma \ref{l-rdm}, we can assume $\sX$ is smooth.

Consider the rational map $f:\sX\dasharrow X$, by Lemma \ref{l-bldm} we know that there is a blow up $\sY \to \sX$ along the indeterminacy locus of $f$, such that there is a morphism $g:\sY\to X$. Moreover, by the construction, we know over $X^{\orb}$,
$$\sY^0:=g^{-1}(X^{\orb})\cong \sX^0. $$

By Lemma \ref{l-rdm}, we know that there is a smooth Deligne-Mumford stack $h:\sY^{\sm}\to \sY$, where $h$ is a representable proper birational morphism which is isomorphic over the smooth locus of $\sY$.  In particular, $h$ is isomorphic over $\sY^0$.

As $\sX$ has finite stabilizer and $\sY^{\sm}\to \sY \to \sX$ is proper, we know that $\sY^{\sm}$ has also finite stabilizer. Thus it follows from \cite{KM97} that $\sY^{\sm}$ admits a coarse moduli space, which we denote by $X^{\orp}$. It has a morphism $f^{\orp}:X^{\orp}\to X$ by the universal property. And we easily check that they satisfy all the properties.

\end{proof}

\noindent
{\bf Acknowledgement:}
We would like to thank Chenyang Xu for communicating the results in the last section to us. The first author would like to thank Professor J. Starr for discussions about stacks.

\end{document}